\documentclass[11pt]{article}

\usepackage[margin=1in]{geometry}  
\usepackage{graphicx}              
\usepackage{amsmath}               
\usepackage{amssymb}               
\usepackage{amsfonts}              
\usepackage{amsthm}                
\usepackage{mathtools}
\usepackage{enumerate}
\usepackage{xparse}
\usepackage[utf8]{inputenc}
\usepackage{color}
\usepackage{authblk}
\usepackage{thm-restate}
\usepackage{hyperref}

\newtheorem{theorem}{Theorem}
\numberwithin{theorem}{section}
\newtheorem{lemma}[theorem]{Lemma}
\newtheorem{claim}[theorem]{Claim}

\newtheorem{conjecture}[theorem]{Conjecture}
\newtheorem{definition}[theorem]{Definition}

\DeclarePairedDelimiter\abs{\lvert}{\rvert}
\DeclarePairedDelimiter\norm{\lVert}{\rVert}
\DeclarePairedDelimiter\floor{\lfloor}{\rfloor}
\DeclarePairedDelimiter\ceil{\lceil}{\rceil}

\makeatletter
\let\oldnorm\norm
\def\norm{\@ifstar{\oldnorm}{\oldnorm*}}
\makeatother

\makeatletter
\let\oldfloor\floor
\def\floor{\@ifstar{\oldfloor}{\oldfloor*}}
\makeatother

\makeatletter
\let\oldceil\ceil
\def\ceil{\@ifstar{\oldceil}{\oldceil*}}
\makeatother

\usepackage{color}
\usepackage[usenames,dvipsnames]{xcolor}

\title{Smaller subgraphs of minimum degree k}

\author[ ]{Frank Mousset,
Andreas Noever, and
Nemanja \v{S}korić}
\affil[ ]{\small Department of Computer Science }
\affil[ ]{ETH Z\"urich, 8092 Z\"urich, Switzerland }
\affil[ ]{\tt {\{moussetf,anoever,nskoric\}@inf.ethz.ch}}
\begin{document}

\maketitle

\abstract{
In 1990 Erd\H{o}s, Faudree, Rousseau and Schelp proved that for $k\geq 2$, every
graph with $n\geq k+1$ vertices and $(k-1)(n-k+2)+\binom{k-2}{2}+1$ edges contains a
subgraph of minimum degree $k$ on at most $n-\sqrt{n}/\sqrt{6k^3}$ vertices.
They conjectured that it is possible to remove at least $\epsilon_k n$ many
vertices and remain with a subgraph of minimum degree $k$, for some
$\epsilon_k>0$. We make progress towards their conjecture by showing that one
can remove at least $\Omega(n/\log n)$ many vertices.
}

\section{Introduction}

It is easy to show that every graph on $n\geq 4$ vertices with at least $2n-2$
edges contains a subgraph of minimum degree $3$. More generally, any graph on
$n\geq k+1$ vertices with at least $t_k(n)\coloneqq
(k-1)(n-k+2)+\binom{k-2}{2}$ edges\footnote{Technically, we may relax the first
condition to $n\geq k-1$, since for $n\in \{k-1,k\}$ the condition on the
number of edges cannot be satisfied. However, there seems no point in doing so.
For $n=k-2$, the statement is wrong.} contains a subgraph of minimum degree $k$,
for all $k\geq 2$. This statement is best possible in two ways: (1) there exist
graphs on $n\geq k+1$ vertices with $t_k(n)-1$ edges which do not contain a
subgraph of minimum degree $k$, and (2) there exist graphs on $n$ vertices with
$t_k(n)$ edges without a subgraph of minimum degree $k$ on fewer than $n$ vertices.
For example the wheel $W(1,n)=K_1+C_{n-1}$ (where $+$ denotes the graph join operation)  has exactly $2n-2$ edges and minimum
degree $3$, but contains no proper induced subgraph with minimum degree $3$. A
similar construction is available for all $k$ (consider the generalized wheel
$W(k-2,n) = K_{k-2}+C_{n-k+2}$).

Erdős conjectured that the presence of even a single additional edge allows one
to find a much smaller subgraph:

\begin{conjecture}[Erd\H{o}s~\cite{FavProblems, ERDOS199053}]
  \label{conj:main}
  For every $k\geq 2$ there exists an $\epsilon_k> 0$ such that every graph on
  $n\geq k+1$ vertices and $t_k(n)+1$ edges contains a subgraph of minimum
  degree $k$ with at most $(1-\epsilon_k)n$ vertices.
\end{conjecture}


%

As far as we know, the only progress on this conjecture is the following theorem
due to Erdős, Faudree, Rousseau and Schelp from 1990:
\begin{theorem}[Erd\H{o}s, Faudree, Rousseau, Schelp~\cite{ERDOS199053}]\label{thm:sqrtn}
  For $k\geq 2$, let $G$ be a graph on $n\geq k+1$ vertices and $t_k(n)+1$
  edges. Then $G$ contains a subgraph of order at most
  $n-\floor*{\sqrt{n/6k^3}}$ and minimum degree at least $k$.
\end{theorem}

Here, we will show that it is possible to replace the $\sqrt{n/6k^3}$ by
$\Omega(n/\log n)$:
\begin{theorem}\label{thm:main}
  For $k\geq 2$, let $G$ be a graph on $n\geq k+1$ vertices and $t_k(n)+1$ edges. Then
  $G$ contains a subgraph of order at most $n-n/(4(k+1)^5 \log_2 n)$ and minimum
  degree at least $k$.
\end{theorem}



We use standard graph theoretic notation. The vertex and edge sets of a graph
$G$ are denoted by $V(G)$ and $E(G)$. We write $v_G$ for the number of vertices
and $e_G$ for the number of edges in $G$. $V_i(G)$ denotes the set of vertices
of $G$ with degree exactly $i$. Similarly $V_{\leq i}(G)$ denotes the set of
vertices of degree at most $i$. For a vertex $v$, we denote its neighborhood by
$\Gamma_G(v)$ and its degree by $\deg_G(v)$ (we omit the subscript $G$ if it is
clear from the context). We write $E_G(A,B)$ (or $E(A,B)$) for the set of edges in $G$
with one endpoint in $A$ and another in $B$. The minimum degree of a graph $G$
is denoted by $\delta(G)$.

\section{Proof of Theorem~\ref{thm:main}}

The proof will use induction on the number of vertices of $G$. In the base
case $n=k+1$ it is easy to check that $t_k(n)+1 = \binom{k+1}{2}+1$, so the
theorem holds vacuously. Assume now that the theorem holds for all graphs $G$
on $k+1\leq n'<n$ vertices.

If $G$ contains a vertex $v$ with $\deg(v)\leq k-1$, then $G-v$ has at least
$t_k(n)+1- (k-1)\geq t_k(n-1)+1$ edges. Thus, by induction, $G-v$ (and
hence $G$) contains a subgraph with minimum degree $k$ and at most
\[ n-1 - \frac{n-1}{4(k+1)^5\log_2(n-1)} \leq n - \frac{n}{4(k+1)^5\log_2 n} \]
vertices, and we are done. From now on, we will assume that $G$
has minimum degree at least $k$.

The rest of the proof is split into two cases depending on the number of
vertices of degree exactly $k$. Already in the proof of \autoref{thm:sqrtn},
Erd\H{o}s, Faudree, Rousseau, and Schelp observed that
Conjecture~\ref{conj:main} holds if the number of vertices of degree $k$ is
not too large:
\begin{lemma}[Lemma 4 in \cite{ERDOS199053}]\label{lemma:fewDegreeK}
  For $k\geq 2 $, let $G$ be a graph on $n$ vertices and $t_k(n)+1$ edges with
  $\delta(G)\geq k$. If for some $0<\alpha <1/(2k)$, $G$ has at most
  $\alpha n$ vertices of degree $k$, then $G$ has a subgraph $H$ of order at
  most $n-(1-2\alpha k)n/(8k^2)$ with $\delta(H)\geq k$.
\end{lemma}

Set $\alpha \coloneqq 1/(2k+2)$. The exact value of $\alpha$ is not too
important, but we need $\alpha < 1/(2k)$ and the given value seems
convenient. If $G$ contains fewer than $\alpha n$ vertices of degree $k$
then, by \autoref{lemma:fewDegreeK}, $G$ contains a subgraph of order at most
\[ n-(1-2\alpha k)n/(8k^2) = n - \frac{n}{8k^2(k+1)}
\leq n - \frac{n}{4(k+1)^5\log_2 n}\]
with minimum degree $k$, and we are done. So from now on assume that $G$
contains at least $\alpha n$ vertices of degree $k$.
The following notion is important for the rest of the proof:

\begin{definition}[Good set]\label{def:good}
  By a \emph{good set}, we mean any set of vertices of $G$ constructed
  according to the following rules:
  \begin{enumerate}
    \item If $v$ has degree $k$, then $\{v\}$ is good.
    \item If $A$ is good and $v\notin A$ is such that all but
      at most $k-1$ neighbors of $v$ belong to $A$, then $A\cup \{v\}$ is good.
    \item If $A$ and $B$ are both good and if $G$ contains
      an edge that meets both $A$ and $B$, then $A\cup B$ is good
      (this is the case if $A\cap B\neq \emptyset$ or if $E(A,B)\neq
      \emptyset$).
  \end{enumerate}
  We say that a good set is \emph{maximal} if it is not properly contained in
  another good set. 
\end{definition}

The relevance of this notion for our problem is partly due to the
following claim.
\begin{claim}\label{claim:good}
  The following statements hold:
  \begin{enumerate}[(i)]
    \item Every good set $C$ intersects at most $(k-1)|C|+1$ edges of $G$.
    \item If $C$ is a good set with $|C|< n-k-1$, then $G-C$ contains a subgraph
      of minimum degree at least $k$.
    \item If $C$ and $C'$ are maximal good sets and $C\neq C'$, then
      $C\cap C'=\emptyset$ and $E_G(C,C') = \emptyset$.
  \end{enumerate}
\end{claim}
\begin{proof}
  The statement (i) is easily proved by induction on the rules
  1, 2, and 3 in the definition of a good set. 
  For (ii), suppose
  that $|C|\leq n-k-1$, so $v_{G-C} = n- |C| \geq k+1$.
  Because $C$ intersects at most $(k-1)\abs{C}+1$ edges, we have
  \[ e_{G-C} \geq e_G-(k-1)\abs{C}-1 \geq t_k(n)- (k-1)\abs{C} = t_k(v_{G-C}), \]
  where the last equality follows from the definition $t_k(n) =
  (k-1)(n-k+2)+\binom{k-2}{2}$. Then (ii) follows because every
  graph $G'$ with $v_{G'}\geq k+1$ and $e_{G'}\geq t_k(v_{G'})$ contains a
  subgraph of minimum of minimum degree $k$.
  Statement (iii) follows immediately from the definition of a good set (rule 3).
\end{proof}

We now handle the case where some good set is very large. Since every good set
is obtained by application of one of the rules given in
Definition~\ref{def:good}, it is clear that if $C$ is a good set of size at
least two, then there is a good subset $C'\subseteq C$ of size $|C|/2\leq
|C'|\leq |C|-1$. In particular, if some good set $C$ has size at least
$n/(k+1)$, then there also exists a good set $C'\subseteq C$ satisfying
$n/(2k+2) \leq |C'|\leq n/(k+1)\leq n-k-1$, where the last inequality holds
since $n> k+1$. Then Claim~\ref{claim:good} (ii) implies that $G-C'$ has a
subgraph of minimum degree $k$, and so $G$ has a subgraph of order at most
$n-n/(2k+2) \leq n - n/(4(k+1)^5\log_2 n)$ with minimum degree $k$, and we are
done. From now on, we may thus assume that every good set has size at most
$n/(k+1)\leq n-k-1$.

To motivate the rest of the proof, we now briefly discuss the proof strategy
that was used by Erd\H{o}s, Faudree, Rousseau, and Schelp
in~\cite{ERDOS199053}. Let $x$ be some parameter with $0 < x < n$, which we
will optimize later. If there is a good set of size at least $x$, then
Claim~\ref{claim:good} (ii) implies that we can find a subgraph of minimum
degree $k$ on at most $n-x$ vertices (the case where the good set is larger
than $n-k-1$ is already excluded). Otherwise, every good set is smaller than
$x$. Now we run an algorithm which
constructs a chain of subgraphs $G = H_0\supseteq H_1\supseteq H_2 \supseteq
\dotsc$ greedily as follows: if $V(H_i)$ contains a maximal good set $C$ such
that $H_i-C$ has minimum degree at least $k$, then we let $H_{i+1}=H_i-C$;
otherwise $H_{i+1}=H_i$. One can show, using the statements in Claim~\ref{claim:good} and
the fact that every maximal good set has size at most $x$, that this
algorithm removes at least $\Omega(n/x)$ good sets, and thus produces a graph
with $n - \Omega(n/x)$ vertices, which has minimum degree at least $k$ by
construction. By setting $x$ to the optimal value $\sqrt{n}$, we thus obtain a
subgraph with minimum degree $k$ of size $n-\Theta(\sqrt{n})$. In our proof, we
avoid the case distinction based on the maximum size of a good set. However, we
also construct a small subgraph of $G$ by removing maximal good sets. More
precisely, we start by choosing a collection $\mathcal C$ of maximal good sets
that covers $\Omega(n/\log n)$ vertices and such that all good sets in
$\mathcal C$ are of comparable sizes; the existence of this collection is
guaranteed by Claim~\ref{claim:collection} further below. Next, we remove a
positive fraction of the sets in $\mathcal C$ from $G$ in such a way that the
remaining graph still contains a subgraph of minimum degree $k$. The main
technical statement that makes this possible is Claim~\ref{claim:cover} below.
Since the sets in $\mathcal C$ all have similar sizes, this means that we
remove $\Omega(n/\log n)$ vertices, completing the proof. We now turn to the
details.

\begin{claim}\label{claim:collection}
  There exists a collection $\mathcal C$ of maximal good sets such that
  \[ n > \sum_{C\in \mathcal C} \abs{C} \geq \frac{\alpha n}{\log_2 n}\]
  and such that for any two $C,C'\in \mathcal C$, we have
  $|C|/2 \leq |C'| \leq 2 |C|$.
\end{claim}
\begin{proof}
  Let $\mathcal C$ denote a maximum-size collection of maximal good sets. Since
  $G$ contains at least $\alpha n$ vertices of degree $k$, and since by
  maximality of $\abs{\mathcal C}$, every such vertex is contained in one of
  the good sets in $\mathcal C$, we have $\sum_{C\in\mathcal C} \abs{C} \geq
  \alpha n$. For every $1\leq i \leq \log_2 n$, let $\mathcal C_i\subseteq
  \mathcal C$ denote the subfamily of all $C\in \mathcal C$ with $2^{i-1}\leq
  \abs C \leq 2^{i}$. By the pigeonhole principle, there exists $i$ such that
  $\sum_{C\in\mathcal C_i} \abs{C} \geq \alpha n/(\log_2 n)$.
  If $\sum_{C\in\mathcal C_i} \abs{C} < n$ then
  $\mathcal C_i$ is a collection with the desired properties.
  Otherwise we remove a single good set, say $C^*$, from $\mathcal C_i$. Since
  each good set has size at most $n/(k+1)$, we have
  \[ \sum_{C\in \mathcal C_i\setminus \{C^*\}} \abs{C} \geq n-n/(k+1) \geq
  \alpha n/(\log_2 n),\]
  where we used $\alpha = 1/(2k+2) \leq (1-1/(k+1))$.
  Then $\mathcal C_i\setminus \{C^*\}$ is a collection with the desired properties.
\end{proof}

Let $\mathcal C$ be a collection as in the statement of
Claim~\ref{claim:collection}. For every $v\in V(G)$, we define $\mathcal
C'(v)\subseteq \mathcal C$ to be the set of all good sets $C\in \mathcal C$
such that $C$ contains a neighbor of $v$ in $G$. Moreover, we let $\mathcal
C(v)\subseteq \mathcal C'(v)$ be any subcollection of size $k+1$ if
$\abs{\mathcal C(v)}> k+1$ and we let $\mathcal C(v) = \mathcal C'(v)$
otherwise.

\begin{claim}\label{claim:cover}
  There is a set $S\subseteq V(G)$ of size at most
  $2|\mathcal C|+k^2$ with the following property:
  for every subfamily $\mathcal C' \subseteq \mathcal C$ 
  such that for all $s\in S$, we have $|\mathcal C' \cap \mathcal C(s)|\leq 1$,
  the graph $G-\bigcup_{C\in
  \mathcal C'}C$ contains a subgraph of minimum degree $k$.
\end{claim}

We postpone the proof of Claim~\ref{claim:cover} to the end of the proof.
For now, assume that we have a set $S$ as in the claim.
Our goal is to find a subcollection $\mathcal C'\subseteq \mathcal C$
of size $\Omega(\abs{\mathcal C})$ containing at most one good set
from every $\mathcal C(s)$.
To find such a collection $\mathcal C'$, we construct an auxiliary `conflict
graph' $A$ on the vertex set $\mathcal C$ by adding a clique on $\mathcal
C(s)$ for every $s\in S$. Note that we are looking for an independent set
in $A$. Because $\abs{\mathcal C(s)}\leq k+1$ holds by construction and
since $|S|\leq 2\abs{\mathcal C}+k^2$, we have
\[ e_A \leq \abs S \binom{k+1}{2} \leq \frac{2(\abs{\mathcal C}+k^2)
(k+1)^2}{2} \leq \frac{\abs{\mathcal C} (2+k^2)(k+1)^2}{2} \leq
\frac{\abs{\mathcal C}((k+1)^4-1)}{2}. \]
By Turán's theorem, any graph on $n$ vertices with at most $cn$ edges contains
an independent set of size at least $n/(2c+1)$. Thus, $A$ contains an
independent set $\mathcal{C'}\subseteq \mathcal C$ of size at least
$\abs{\mathcal C}/(k+1)^4$.
Because any two $C,C'\in \mathcal C$ satisfy $|C|/2\leq |C'|\leq 2|C|$ and
since $\sum_{C\in\mathcal C} \abs{C} \ge \alpha n/(\log_2 n)$, we have
\[ \sum_{C\in \mathcal C'} \abs{C}\geq \frac{\alpha n}{2(k+1)^4\log_2 n}. \]
Then $G'\coloneqq G-\bigcup_{C\in \mathcal C'} C$ has at most
$n-\alpha n/(2(k+1)^4 \log_2 n)$ vertices and contains a subgraph of minimum
degree $k$, by the defining property of $S$. Recalling that $\alpha =
1/(2k+2)$, this completes the proof of the theorem.
It remains to prove Claim~\ref{claim:cover}.

\subsection{Proof of Claim~\ref{claim:cover}}

For the proof of the claim, we need the following definition and lemma.

\begin{definition}[$(H,S,k)$-cover]
  Suppose that $H$ is a graph and that $S\subseteq V(H)$ is a subset of its
  vertices. Given $k\geq 2$, a graph $\tilde H$ is called an \emph{$(H,
  S,k)$-cover} if it contains $H$ as a subgraph and if 
  $V_{\leq k-1}(\tilde H)\subseteq V(H)\setminus S$.
\end{definition}

\begin{lemma}\label{lemma:cover}
  Suppose that $H$ is a graph that does not contain a subgraph of minimum
  degree $k$, for $k\geq 2$. Then there exists a subset $S \subseteq V_{\leq
  k-1}(H)$ of cardinality at most $2(k-1)v_H-2e_H$ such that every
  $(H,S,k)$-cover contains a subgraph of minimum degree $k$.
\end{lemma}

\begin{proof}
  For a graph $H$, we define the function
  \begin{equation} \label{eq:defphi}
  \phi(H)\coloneqq 2(k-1)v_H-2e_H-\sum_{w \in V_{\le k - 1}(H)} (k - 1 - \deg_H(w)).  
  \end{equation}
  By induction on the number of vertices of $H$, we will prove the following
  statement, which is slightly stronger than the claim of the lemma: if $H$ has
  no subgraph of minimum degree $k$, then there is a subset $S\subseteq V_{\leq
  k-1}(H)$ of size at most $\phi(H)$ such that every $(H,S,k)$-cover contains a
  subgraph of minimum degree $k$.

  In the base case $v_H=1$, we can let $S$ be the set containing the single
  vertex of $H$. In this case, we have $|S| \leq k-1=\phi(H)$, and every
  $(H,S,k)$-cover has minimum degree at least $k$ by definition.

  If $v_H \geq 2$ then the fact that $H$ does not contain a subgraph of minimum
  degree $k$ implies that there is a vertex $v$ with $\deg(v) \leq k-1$. Let
  us write $H'\coloneqq H-v$. Then $H'$ is a graph without a subgraph of
  minimum degree $k$, and it has fewer vertices than $H$. Hence, by induction,
  it contains a set $S'\subseteq V_{\leq k-1}(H')$ such that every
  $(H',S',k)$-cover has a subgraph of minimum degree $k$.
  We now define the set $S$ by
  \begin{equation*}
    S\coloneqq(S'\cup I_v) \setminus V_k(H),
  \end{equation*}
  where
  \begin{equation*}
    I_v =
    \begin{cases}
      \{v\} & \text{if $\deg(v)\leq k-2$ or $\Gamma(v)\cap S' \neq \emptyset$,} \\
      \emptyset & \text{otherwise.}
    \end{cases}
  \end{equation*}

  Note that from the definition of $S$ it follows that if $v \notin S$ then we
  have $S = S'$. Furthermore,  since $S'\subseteq V_{\leq k-1}(H')$ and by
  definition $S$ does not contain vertices of degree $k$, we also have $S
  \subseteq
  V_{\leq k-1}(H)$. 
  To check that $S$ is not too big, note that 
  \begin{align*}
  \phi(H)-\phi(H') & = 2 (k - 1) - 2 \deg_H(v) -
   \big( k - 1 - \deg_H(v)  - |\Gamma(v) \cap V_{\le k - 1}(H)| \big)   \\ 
   & = (k - 1) - \deg_H(v) +| V_{\le k - 1}(H) \cap \Gamma(v)|  \ge 0.
  \end{align*}
  Therefore if $v\notin S$ then $\abs{S}\leq \abs{S'}\leq \phi(H')\leq
  \phi(H)$. If $v\in S$ then we distinguish two cases:
  \begin{enumerate}
    \item If $\deg(v)=k-1$ and $\abs {V_{\leq k-1}(H)\cap \Gamma(v)} =0$ then
      there must exist a vertex $u\in S' \cap \Gamma(v)$ with degree at least
      $k$. By definition of $S'$, $u$ has degree at most $k-1$ in $H'=H-v$ and
      thus $u$ has degree exactly $k$ in $H$. In particular, we have $u\in
      S'\setminus S$ and thus $\abs{S} \leq \abs{S'} \leq \phi(H')\leq
      \phi(H)$.
    \item Otherwise, we either have
      $\deg(v)\leq k-2$ or we have $\deg(v)=k-1$ and $\abs {V_{\leq
      k-1}(H)\cap \Gamma(v)} \geq 1$. In both cases, we have
      $\phi(H)-\phi(H')\geq 1$ and so $\abs{S}\leq \abs{S'}+1 \leq \phi(H')+1
      \leq \phi(H)$.
  \end{enumerate}

  Now suppose that $\tilde H\supseteq H$ is an $(H,S,k)$-cover.
  We claim that then either $\tilde H$ is an $(H',S',k)$-cover or $\tilde H-v$ is
  an $(H',S',k)$-cover. In both cases, the induction hypothesis implies that
  $\tilde H$ has a subgraph of minimum degree $k$.
  To show this, we first recall that by the definition of an
  $(H,S,k)$-cover, we have $V_{\leq k-1}(\tilde H)\subseteq V(H)\setminus S$.
  We distinguish two cases. If $\deg_{\tilde H}(v)\geq k$ then actually
  $V_{\leq k-1}(\tilde H)\subseteq V(H)\setminus (S\cup \{v\}) =
  V(H')\setminus S$. Moreover, by construction of $S$, all vertices of $V(H')$
  that belong to $S'\setminus S$ must have degree $k$ in $H$ (and so degree at
  least $k$ in $\tilde H$). Therefore we have $V_{\leq k-1}(\tilde H)\subseteq
  V(H')\setminus S'$. In other words, $\tilde H$ is $(H',S',k)$-cover.
  Otherwise, we have
  $\deg_{\tilde H}(v)\leq k-1$. Then $v\notin S$ and thus
  $\deg_{H}(v)=k-1=\deg_{\tilde H}(v)$, $\Gamma_H(v)\cap S'=\emptyset$
  and $S'= S$.
  It
  moreover follows that $\Gamma_{\tilde H}(v) = \Gamma_H(v) \subseteq
  V(H')\setminus S'$.
  These  observations show that
  \begin{align*}
  V_{\leq k-1}(\tilde H-v) & \subseteq (V_{\leq k-1}(\tilde H)\setminus \{v\})
  \cup \Gamma_{\tilde H}(v) \\
  & \subseteq
   (V(H)\setminus (S\cup \{v\}))
  \cup \Gamma_{\tilde H}(v)  \\
  & = V(H')\setminus S'.
  \end{align*}
  Thus $\tilde H-v$ is an $(H',S',k)$-cover, completing the proof.
\end{proof}

Using this lemma, we now prove Claim~\ref{claim:cover}.
Consider the graph $H\coloneqq G-\bigcup_{C\in\mathcal C} C$ obtained
by removing all sets in $\mathcal C$ from $G$.
We have $v_H = n - \sum_{C\in \mathcal C} |C|>0$. By Claim~\ref{claim:good} (i),
every good set $C$ intersects at most $(k-1)|C|+1$ edges, and so
\begin{equation}\label{eq:eh}
  e_H \geq t_k(n)+1-\sum_{C\in\mathcal C}\big((k-1)\abs C + 1\big) =
  t_k(v_H)-\abs{\mathcal C} + 1.
\end{equation}
If $H$ contains a subgraph of minimum degree $k$, then we are done because
we can simply choose $S=\emptyset$. Otherwise, we apply \autoref{lemma:cover}
to $H$ to obtain a set $S \subseteq V_{\leq k-1}(H)$ of size
\[
  |S|\leq 2(k-1)v_H-2e_H \leq 2(k-1)(k-2) - 2\binom{k-2}{2} + 2\abs{\mathcal
  C_i} -2 \leq 2\abs{\mathcal C_i} + k^2
  \]
(using~\eqref{eq:eh} and the definition of $t_k(v_H)$ to bound $e_H$)
such that every $(H,S,k)$-cover contains a subgraph of minimum degree $k$.
To complete the proof of the claim, suppose that
$\mathcal C'\subseteq \mathcal C$ contains at most one set from each
set $\mathcal C(s)$ where $s\in S$. It is enough to show that the graph 
$G - \bigcup_{C\in \mathcal C'} C$ is an $(H,S,k)$-cover.
Note first that since each element of $\mathcal C$ is a
maximal good set, the elements of $\mathcal C$ are pairwise disjoint and for
any two distinct $C,C'\in \mathcal C$, there are no edges between $C$ and
$C'$ in $G$. 
Since $G$ has minimum degree at least $k$ by assumption,
this means in particular that 
$V_{\leq k-1}(G-\bigcup_{C\in\mathcal C'} C) \subseteq V(H)$.
Furthermore, every $s\in S$ has degree at least $k$ in $G -
\bigcup_{C\in \mathcal C'} C$, for one of the following two reasons:
\begin{enumerate}
\item Either $\mathcal C' \cap \mathcal C(s) = \emptyset$ or $|\mathcal C(s)
  \setminus \mathcal C'| \ge k$ in which case the claim directly follows.
\item Or $\mathcal C(s) =\mathcal C'(s)$ and
  we have removed exactly one good set from $\mathcal C(s)$, say
  $\mathcal C' \cap \mathcal C(s)=\{\tilde C\}$. Then the degree of $s$ in $G
    - \tilde C$ (which equals the degree in $G -
    \bigcup_{C\in \mathcal C'} C$)
    must be at least $k$ by the maximality of $\tilde C$.
\end{enumerate}
Thus $V_{\leq k-1}(G-\bigcup_{C\in\mathcal C'} C) \subseteq V(H)\setminus S$,
so $G - \bigcup_{C\in \mathcal C'} C$ is an $(H,S,k)$-cover. This
completes the proof of the claim.

\section{Acknowledgements}

We are grateful to Rajko Nenadov and Yury Person for fruitful discussions.

\bibliographystyle{abbrv}
\bibliography{mindeg}

\end{document}